\documentclass[leqno,12pt]{amsart} 
\usepackage{amssymb}
\usepackage{amsmath}
\usepackage{amsthm}
\usepackage{amscd}
\usepackage{graphicx}
\usepackage[dvips]{color}
\usepackage[all]{xy}
\theoremstyle{plain} 
\newtheorem{theorem}{\indent\sc Theorem}[section]
\newtheorem{lemma}[theorem]{\indent\sc Lemma}
\newtheorem{corollary}[theorem]{\indent\sc Corollary}

\theoremstyle{definition} 
\newtheorem{definition}[theorem]{\indent\sc Definition}
\newtheorem{remark}[theorem]{\indent\sc Remark}
\newtheorem{example}[theorem]{\indent\sc Example}

\begin{document}

\title[On the minimality of canonically attached singular Hermitian metrics]
{On the minimality of canonically attached singular Hermitian metrics on
certain nef line bundles} 

\author[T. Koike]{Takayuki Koike} 

%
\address{
Graduate School of Mathematical Sciences, The University of Tokyo \endgraf
3-8-1 Komaba, Meguro-ku, Tokyo, 153-8914 \endgraf
Japan
}
\email{tkoike@ms.u-tokyo.ac.jp}

\maketitle

\begin{abstract}
We apply Ueda theory to a study of singular Hermitian metrics of a (strictly) nef line bundle $L$. 
Especially we study minimal singular metrics of $L$, metrics of $L$ with the mildest singularities
among singular Hermitian metrics of $L$ whose local weights are plurisubharmonic. 
In some situations, 
we determine a minimal singular metric of $L$. 
As an application, we give new examples of (strictly) nef line bundles which admit no smooth Hermitian metric with semi-positive curvature. 
\end{abstract}

\section{Introduction}
Our interest is the singularity of minimal singular metrics on nef line bundles over smooth projective surfaces. 
Minimal singular metrics of a line bundle $L$ are metrics of $L$ with the mildest singularities
among singular Hermitian metrics of $L$ whose local weights are plurisubharmonic. Minimal singular metrics were introduced in [DPS00, 1.4] as a (weak) analytic analogue
of the Zariski decomposition, and always exist when $L$ is pseudo-effective ([DPS00, 1.5]). 
In this paper, we mainly consider a topologically trivial line bundle on a surface which is defined by a smooth embedded curve with a neighborhood of non-trivial complex structure (rigorously speaking, when the curve is {\it of finite type} in the sense of \cite[p. 589]{U}. See Definition \ref{type} here). 
The goal of this paper is to determine a minimal singular metric of such a line bundle. 
The main theorem is as follows.

\begin{theorem}\label{main_theorem}
Let $X$ be a smooth complex surface and $C\subset X$ be an embedded smooth compact complex curve. 
Assume $(C^2)=0$ and the pair $(C, X)$ is of finite type. 
Then a singular Hermitian metric $|f_C|^{-2}$ on $\mathcal{O}_X(C)$ has minimal singularities, where $f_C\in H^0(X, \mathcal{O}_X(C))$ is a section whose zero divisor is $C$. 
Especially, $\mathcal{O}_X(C)$ is nef, however it admits no smooth Hermitian metric with semi-positive curvature. 
\end{theorem}

We can apply Theorem \ref{main_theorem} to a line bundle defined by the section of a certain ruled surface. 
We are interested in such a situation since it includes the example of Demailly, Peternell, and Schneider \cite[1.7]{DPS94}, 
which is constructed as a nef line bundle which admits no smooth Hermitian metric with semi-positive curvature. 
We determine a minimal singular metric of such a line bundle, 
and in particular give a generalization of their result. 

\begin{corollary}[Example \ref{gen_dps_eg}]\label{gen_dps}
Let $C$ be a smooth projective curve and  $L$ be a topologically trivial line bundle on $C$ such that $H^1(C, L)\not=0$. 
For example, this condition is satisfied for all topologically trivial line bundle when $C$ is a curve of genus greater than $1$. 
Take a non-zero element $\xi\in H^1(C, L)$ 
and let $E$ be the rank-two vector bundle on $C$ which appears in the non-splitting exact sequence 
$0\to L\to E\to\mathcal{O}_C\to 0$ corresponding to the class $\xi$. 
Denote by $X$ a rational surface $\mathbb{P}(E)$ over $C$ and by $D$ the section of the map $\mathbb{P}(E)\to\mathcal{O}_C$. 
Then a singular Hermitian metric $|f_D|^{-2}$ on $\mathcal{O}_X(D)$ has minimal singularities, where 
$f_D\in H^0(X, \mathcal{O}_X(D))$ is a section whose zero divisor is $D$. 
Especially the line bundle $\mathcal{O}_X(D)$ is nef, however it admits no smooth Hermitian metric with semi-positive curvature. 
\end{corollary}

The example in \cite[1.7]{DPS94} is a special case of Example \ref{gen_dps_eg}, in which $C$ is an elliptic curve and $L=\mathcal{O}_C$. 
We remark that they give a minimal singular metric of $\mathcal{O}_X(D)$ for this case by determining all singular Hermitian metrics on this line bundle with semi-positive curvature. 
Our proof is based on a completely different idea. 

We are also interested in minimal singular metrics of strictly nef and non semi-ample line bundles, where we say a line bundle $L$ on a variety $X$ is {\it strictly nef} if the intersection number $(L. C)$ is positive for all compact curve $C\subset X$. 
The following examples are strictly nef and non semi-ample line bundles. 
Example \ref{fujino_eg} (1) is so-called Mumford's example, which admits a smooth Hermitian metric with semi-positive curvature. 
Example \ref{fujino_eg} (2) is given by Fujino. 
He proposed a question whether this example  admits a smooth Hermitian metric with semi-positive curvature \cite[5.10]{F}. 

\begin{example}\label{fujino_eg}\ \\
\indent
(1) ({\cite[10.6]{H}}) Let $\widetilde{C}$ be a smooth compact curve of genus greater than $1$. 
Mumford showed that there exists a rank-two vector bundle $F$ on $\widetilde{C}$ such that its degree is equal to zero and its symmetric powers $S^m(F)$ are stable for all $m\geq 1$. 
Let $Y$ be the total space of a projective space bundle $\mathbb{P}(F)$ and $L_Y$ be the relative hyperplane bundle $\mathcal{O}_Y(1)$. 
Then $L_Y$ is a strictly nef and non semi-ample line bundle. 
In this situation, we can construct a smooth Hermitian metric $h_{L_Y}$ on $L_Y$ with semi-positive curvature as a induced metric from the unitary-flat metric of the stable bundle $F$ (see Remark \ref{mumford_rmk}). 

(2) (A variant of {\cite[5.9]{F}}.) 
Fix a smooth compact curve $\widetilde{C}$ of genus greater than $1$. 
Since $H^1(\widetilde{C}, \mathcal{O}_{\widetilde{C}})\not=0$, we can take a non-zero element $\xi\in H^1(\widetilde{C}, \mathcal{O}_{\widetilde{C}})$. 
Let $E$ be the rank-two vector bundle on $\widetilde{C}$ which appears in the non-splitting exact sequence 
$0\to \mathcal{O}_{\widetilde{C}}\to E\to\mathcal{O}_{\widetilde{C}}\to 0$ corresponding to the class $\xi$. 
Denote by $\widetilde{X}$ the rational surface $\mathbb{P}(E)$ over $\widetilde{C}$, by $\widetilde{D}$ the section of the map $\mathbb{P}(E)\to\mathcal{O}_{\widetilde{C}}$, 
and by $\widetilde{Y}$ the fiber product $\widetilde{X}\times_{\widetilde{C}}Y$, 
where $Y$ is that in Example \ref{fujino_eg} (1). 
\[\xymatrix{
\widetilde{Y}=\widetilde{X}\times_{\widetilde{C}}Y\ar[r]^{p_1} \ar[d]^{p_2}  &\widetilde{X}=\mathbb{P}(E) \ar[d]\\
Y \ar[r] &\widetilde{C}     \\
}\]
In this situation, it can be shown that the line bundle $\widetilde{L}:=\mathcal{O}_{\widetilde{Y}}(\widetilde{D}\times_{\widetilde{C}}Y)\otimes p_2^*L_Y$ is a strictly nef and non semi-ample line bundle, where $p_2\colon\widetilde{Y}\to Y$ is the second projection and $L_Y$ is that in Example \ref{fujino_eg} (1). 
\end{example}

We also determine a minimal singular metric of $\widetilde{L}$ in Example \ref{fujino_eg} (2),
and in particular give the answer to the above Fujino's question. 

\begin{corollary}\label{fujino_answer}
Let $\widetilde{L}=\mathcal{O}_{\widetilde{Y}}(\widetilde{D}\times_{\widetilde{C}}Y)\otimes p_2^*L_Y$ be that in Example \ref{fujino_eg} (2), 
$f$ a global section of $\mathcal{O}_{\widetilde{Y}}(\widetilde{D}\times_{\widetilde{C}}Y)$ whose zero divisor is $\widetilde{D}\times_{\widetilde{C}}Y$, 
and $h_{L_Y}$ be a smooth Hermitian metric on $L_Y$ with semi-positive curvature (see Remark \ref{mumford_rmk} for the existence of such a metric on $L_Y$). 
Then the metric $|f|^{-2}\otimes(p_2^*h_{L_Y})$ is a minimal singular metric of $\widetilde{L}$. 
In particular, $\widetilde{L}$ is {\rm not} semi-positive. 
\end{corollary}

We will prove this corollary for more general situation (see Theorem \ref{countereg}). 

The organization of the paper is as follows. 
In \S2, we define some concepts on singular Hermitian metrics over possibly non-compact complex manifolds. 
In \S3, we review Ueda theory, prove Theorem \ref{main_theorem}, and apply it on some examples. 
In \S4, we prove Corollary \ref{fujino_answer} in more general form (Theorem \ref{countereg}). 

\vskip3mm
{\bf Acknowledgment. } 
The author would like to thank his supervisor Prof. Shigeharu Takayama whose enormous support and insightful comments were invaluable during the course of his study. 
He is supported by the Grant-in-Aid for Scientific Research (KAK-
ENHI No.25-2869) and the Grant-in-Aid for JSPS fellows. 
This work is supported by the Program for Leading Graduate
Schools, MEXT, Japan. 

\section{Line bundles on possibly non-compact complex manifolds}
In this section, we define some concepts on singular Hermitian metrics over possibly non-compact complex manifolds. 
Let $X$ be a complex manifold and $L$ be a line bundle on $X$ (In this paper, ``line bundle" always stands for a holomorphic line bundle). 
Let $h$ be a singular Hermitian metric of $L$ (for the definition of the singular Hermitian metric, see \cite[3.12]{D}). 
Then, for each local trivialization of $L$ on an open set of $X$, 
``the inner product" defined by $h$ can be written as
\[
\langle\xi, \eta\rangle_z=e^{-\psi(z)}\xi\overline{\eta}
\]
where $z$ is a point in the open set, $\xi$ and $\eta$ are points in $\mathbb{C}$, which we regard as the $z$-fiber of $L$, 
and $\psi$ is a locally integrable function defined on the open set, which we call the {\it local weight} of $h$. 
Here we remark that it is known that the local currents $\sqrt{-1}\partial\overline{\partial} \psi$ glue together to define the curvature current associated to $h$. 
We denote it by $\sqrt{-1}\Theta_h$. 

Next, let us recall how to compare the singularity of two psh functions. 

\begin{definition}\label{sim_sing} (\cite[1.4]{DPS00})
Let $\varphi$ and $\psi$ be psh functions defined on a neighborhood of $x\in X$. 
We say $\psi$ {\it is less singular than} $\varphi$ and write $\varphi\prec_{\rm sing}\psi$ at $x$ when there exists a constant $C$ such that the inequality $\varphi\leq \psi+C$ holds for each point sufficiently near to $x$. 
We denote $\varphi\sim_{\rm sing}\psi$ at $x$ if $\varphi\prec_{\rm sing}\psi$ and $\varphi\succ_{\rm sing}\psi$ holds at $x$. 
\end{definition}

We define the minimal singular metric as follows. 

\begin{definition}
Let $h_{\rm min}$ be a singular Hermitian metric of $L$ which satisfies $\sqrt{-1}\Theta_{h_{\rm min}}\geq 0$. 
We call $h_{\rm min}$ a {\it minimal singular metric} 
if $\psi \prec_{\rm sing}\varphi_{\rm min}$ holds at any point $x\in X$ 
for all singular Hermitian metric $h$ satisfying $\sqrt{-1}\Theta_h\geq 0$, 
where $\varphi_{\rm min}$ and $\psi$ stands for the local weight functions of $h_{\rm min}$ and $h$, respectively. 
\end{definition}

It is known that every pseudo-effective line bundles on compact complex manifolds admit minimal singular metrics \cite[1.5]{DPS00}. 
We remark that, though the minimal singular metric is not unique, but it is unique up to the relation $\sim_{\rm sing}$ if it exists. 

\section{Ueda theory and Proof of Theorem \ref{main_theorem}}

First, let us review Ueda theory along \cite[\S2]{U} for a smooth complex surface $X$ and an embedded smooth compact complex curve $C\subset X$ with $(C^2)=0$. 
For the normal bundle $N_{C/X}$ is a topologically trivial line bundle on a compact K\"ahler manifold, it admits a flat structure; i.e. 
\[N_{C/X}=\{(U_j\cap U_k, t_{jk})\}_{jk}\]
holds in $H^1(C, \mathcal{O}_C^*)$ for a sufficiently fine open covering $\{U_j\}_j$ and some constants $t_{jk}\in U(1):=\{t\in\mathbb{C}\mid |t|=1\}$. 
Let $V$ be a sufficiently small tubular neighborhood of $C$ in $X$ and $\{V_j\}_j$ be a sufficiently fine open covering of $V$. Without loss of generality, we may assume that the index sets of $\{U_j\}_j$ and $\{V_j\}_j$ coincide and $V_j\cap C=U_j$ holds. 
We choose local coordinates $(z_j, x_j)$ of $V_j$ satisfying conditions that $x_j$ is a coordinate of $U_j$, $\{z_j=0\}=U_j$ holds on $V_j$, and that $z_k/z_j\equiv t_{jk}$ holds on $U_j\cap U_k$ for all $j$ and $k$. 
Let $n$ be a positive integer. 
We call $\{(V_j, z_j, x_j)\}_j$ {\it a system of type $n$} if ${\rm mult}_{U_j\cap U_k}(t_{jk}z_j-z_k)\geq n+1$ holds on each $V_j\cap V_k$. When there exists a system $\{(V_j, z_j, x_j)\}_j$ of type $n$, the Taylor expansion of $t_{jk}z_j$ for the variable $z_k$ on $V_j\cap V_k$ can be written in the form 
\[
t_{jk}z_j=z_k+f_{jk}(x_k)z_k^{n+1}+\cdots
\]
for some holomorphic function $f_{jk}$ defined on $U_j\cap U_k$. 
Here we remark that, for all $m>n$, a system $\{(V_j, z_j, x_j)\}_j$ of type $m$ is also a system of type $n$ and in this case the above $f_{jk}$ is the constant function $0$. 
It is known that $\{(U_j\cap U_k, f_{jk}|_{U_j\cap U_k})\}_{jk}$ satisfies the cocycle condition \cite[p. 588]{U}. 

\begin{definition}
Suppose that there exists a system of type $n$. 
Then the cohomology class 
\[
u_n(C, X):=\{(U_j\cap U_k, f_{jk}|_{U_j\cap U_k})\}_{jk}\in H^1(C, N_{C/X}^{-n})
\]
is called the $n$-th Ueda class of the pair $(C, X)$. 
\end{definition}

The $n$-th Ueda class does not depend on the choice of local coordinates system (\cite[1.3]{N}). 
It is known that $u_n(C, X)=0$ if and only if there exists a system of type $n+1$. 
Thus only one phenomenon of the following occurs. 
\begin{description}
\item[(1)] There exists an integer $n\in\mathbb{Z}_{>0}$ such that $u_m(C, X)$ can be defined only when $m\leq n$, $u_m(C, X)=0$ holds for all $m<n$, and $u_n(C, X)\not=0$ holds. 
\item[(2)] For every integer $n\in\mathbb{Z}_{>0}$, $u_n(C, X)$ can be defined and it is equal to zero. 
\end{description}

\begin{definition}[{\cite[p. 589]{U}}]\label{type}
We denote ${\rm type}\,(C, X)=n$ and say that the pair $(C, X)$ is of finite type when (1) above occurs.  
In the other case, we denote ${\rm type}\,(C, X)=\infty$ and say that the pair $(C, X)$ is of infinite type. 
\end{definition}

\begin{theorem}[{\cite[Theorem 2]{U}}]\label{ueda_theorem} 
Let V be a neighborhood of $C$ in $X$ and $\Psi$ be a psh
function on $V\setminus C$. 
Assume that ${\rm type}\,(C, X)=n$ holds and that there exists a positive real number $a<n$ such that $\Psi (p)=o({\rm dist}\,(p, C)^{-a})$ as p approaches $C$, where ${\rm dist}\,(p, C)$ is the Euclidean distance between $p$ and $C$.
Then there exists a neighborhood $V_0$ of $C$
such that $\Psi$ is constant on $V_0\setminus C$. 
\end{theorem}

Now we prove Theorem \ref{main_theorem} as an application of Theorem \ref{ueda_theorem}. 

\begin{lemma}\label{ueda_lemma}
Let $X$ be a smooth complex surface, $C\subset X$ an embedded smooth compact complex curve, $f_C\in H^0(X, \mathcal{O}_X(C))$ a section whose zero divisor is $C$, and $h$ be a singular Hermitian metric of $\mathcal{O}_X(C)$ with semi-positive curvature. 
Assume $(C^2)=0$, ${\rm type}\,(C, X)<\infty$, 
and that the local weight functions of $h$ are less singular than the psh function $\log |f_C|^2$. 
Then there exists a positive number $M$ such that $h=M|f_C|^{-2}$ holds on a neighborhood of $C$ in $X$. 
\end{lemma}

\begin{proof}
Let us fix a sufficiently small neighborhood $V$ of $C$ in $X$ and consider a function 
$
\Psi:=-\log |f_C|_h^2
$
defined on $V\setminus C$. 
In the following, we sometimes restrict our selves to a sufficiently small open neighborhood of each point of $V$. 
On the open neighborhood, we fix a local trivialization of $\mathcal{O}_X(C)$, and by using a local weight function $\varphi$ of $h$, 
we denote $|f_C|_h^2$ locally by $|f_C|^2e^{-\varphi}$, where we regard $f_C$ as a locally defined holomorphic function corresponding to the section $f_C$. 

First we check that this function is psh. 
By using the above notation, it is clear that the equation $\Psi=\varphi-\log |f_C|^2$ holds locally. Since $\log |f_C|^2$ is harmonic on $V\setminus C$ and $\varphi$ is psh by assumption, we can conclude that $\Psi$ is a psh function. 

Next, we evaluate the divergence of $\Psi$ near $C$. 
Since $\log|f_C|^2\prec_{\rm sing}\varphi$, $\Psi$ is bounded from below. 
Thus all we have to do is to evaluate the divergence of $\Psi$ to $+\infty$. 
We use local coordinates system $(z, x)$ of $V$ such that $x$ is a local coordinate of $C$ and the equation $\{z=0\}=C$ holds locally. 
Without loss of generality, we may assume $f_C(z, x)=z$ holds on this locus. 
For $\varphi$ is a psh function, it is locally bounded from above. 
Thus 
\[
\Psi(z, x)=\varphi(z, x)-\log |z|^2\leq C-\log |z|^2= o(1/|z|^{1/2})\ \ {\rm as}\ |z|\to 0. 
\]
holds. 
Therefore we can apply \cite[Theorem 2]{U} and thus, after making the neighborhood $V$ smaller, we can conclude that $\Psi$ is a constant map, which shows the lemma. 
\end{proof}

\begin{proof}[Proof of Theorem \ref{main_theorem}]
Let $h$ be a smooth Hermitian metric of $\mathcal{O}_X(C)$ with semi-positive curvature. 
We denote by $\widetilde{h}$ the new metric $\min\{h, |f_C|^{-2}\}$. 
Clearly we obtain that the local weight function of $\widetilde{h}$ is $\max\{\log|f_C|^2, \varphi\}$, where $\varphi$ is the local weight function of $h$.  Since this function is the maximum of two psh functions, $\widetilde{h}$ also has the semi-positive curvature. 
For $\log |f_C|^2\prec_{\rm sing}\max\{\log|f_C|^2, \varphi\}$, we can apply Lemma \ref{ueda_lemma} and thus we obtain that $\log |f_C|^2\sim_{\rm sing}\max\{\log|f_C|^2, \varphi\}\succ_{\rm sing}\varphi$ holds. 
\end{proof}

In the rest of this section, we give some examples of nef but not semi-positive line bundles over smooth projective varieties as applications of Theorem \ref{main_theorem}. 

\begin{example}\label{gen_dps_eg}
Let $C$ be an smooth projective curve of positive genus and $E$ be the rank-$2$ vector bundle on $C$ appears in the exact sequence 
$0\to L\to E\to\mathcal{O}_C\to 0$, where $L$ is a topologically trivial line bundle on $C$ such that $H^1(C, L)\not=0$. 
Let us define that $X$ is the rational surface $\mathbb{P}(E)$ over $C$ and $D:=\mathbb{P}(\mathcal{O}_C)\subset \mathbb{P}(E)$ be the section of the map $\mathbb{P}(E)\to C$. 
We denote by $\pi\colon X\to C$ the canonical surjection. 
From simple computations, we can conclude that the nomal bundle $N_{D/X}$ is isomorphic to $(\pi|_D)^*L^{-1}$ 
and thus it is topologically trivial. 
Therefore we can define Ueda classes for the pair $(D, X)$. 
In our case, we obtain that the first Ueda class $u_1(D, X)\in H^1(D, N_{D/X}^{-1})$ coincides with cohomology class 
$\{E\}\in{\rm Ext}^1(\mathcal{O}_C, L)=H^1(C, L)$ via the isomorphism $\pi|_D\colon D\to C$ (see \cite[Proposition 7.6]{N}). 
Thus the first Ueda class $u_1(D, X)$ vanishes if and only if $E=\mathcal{O}_C\oplus L$, and in this case, 
the line bundle $\mathcal{O}_X(D)$ is semi-positive. 

Assume that the exact sequence 
$0\to L\to E\to\mathcal{O}_C\to 0$ 
is non-splitting. 
In this case, the first Ueda class $u_1(D, X)$ does not vanish and thus ${\rm type}\,(D, X)=1<\infty$ and we can apply Theorem \ref{main_theorem}. 
As a result, we can conclude that $|f_{{D}}|^{-2}$ defines a minimal singular metric of $\mathcal{O}_{{X}}({D})$, where $f_{{D}}\in H^0(X, \mathcal{O}_{{X}}({D}))$ is a section whose zero divisor is ${D}$. 

When $C$ is a smooth elliptic curve and $L=\mathcal{O}_C$, the minimal singularity of the above singular Hermitian metric $|f_{{D}}|^{-2}$ is known by 
Demailly, Peternell, and Schneider \cite[1.7]{DPS94}. 
We also remark that, in this case, they more precisely shows that for every singular Hermitian metric $h$ of $\mathcal{O}_X(D)$ with semi-positive curvature, there exists a positive constant $M$ such that $h=M|f_D|^{-2}$ holds. 
\end{example}

\begin{example}\label{neeman_eg}
Let $C$ be a smooth compact curve of genus $2$ and $Y$ be its Jacobian. 
Fix two points $p, q\in C$ conjugate to each other by the hyperelliptic involution. 
We denote by $X$ the blow-up of $Y$ at $p$ and $q$ and by $D$ the the strict transformation of $C$. 
According to Neeman \cite[10.5]{N}, ${\rm type}\,(D, X)=1<\infty$ holds for this $D$ and $X$. 
Therefore we can apply Theorem \ref{main_theorem} and conclude that $|f_{{D}}|^{-2}$ defines a minimal singular metric of $\mathcal{O}_{{X}}({D})$, where $f_{{D}}\in H^0(X, \mathcal{O}_{{X}}({D}))$ is a section whose zero divisor is ${D}$. 
\end{example}

\section{Strictly nef line bundle which admits no smooth Hermitian metric with semi-positive curvature}
In this section, we prove the following 

\begin{theorem}\label{countereg}
Let $\widetilde{C}, Y$, and $L_Y$ be those in Example \ref{fujino_eg} (1), $\widetilde{X}\to\widetilde{C}$ be a smooth holomorphic fiber bundle over $\widetilde{C}$ ($\widetilde{X}$ need not to be that in Example \ref{fujino_eg} (2) and also need not to be compact) 
and $L_{\widetilde{X}}$ is a pseudo-effective line bundle on $\widetilde{X}$. 
Denote by $\widetilde{Y}$ the fiber product $\widetilde{X}\times_{\widetilde{C}}Y$ and 
by $\widetilde{L}$ the line bundle $p_1^*L_{\widetilde{X}}\otimes p_2^*L_Y$, 
where $p_1\colon\widetilde{Y}\to\widetilde{X}$ and $p_2\colon\widetilde{Y}\to Y$ is the first and second projection, respectively. 
Assume that $L_{\widetilde{X}}$ admits a minimal singular metric $h_{L_{\widetilde{X}}}$. 
Then the metric $(p_1^*h_{L_{\widetilde{X}}})\otimes(p_2^*h_{L_Y})$ is a minimal singular metric of $\widetilde{L}$, 
where $h_{L_Y}$ is that in Example \ref{fujino_eg} (1) (see Remark \ref{mumford_rmk}). 
\end{theorem}

Here a ``smooth holomorphic fiber bundle over $\widetilde{C}$" means a holomorphic map from smooth variety onto $\widetilde{C}$ which satisfies that, for each point $x\in \widetilde{C}$, there exists an open neighborhood $\widetilde{U}\subset\widetilde{C}$ of $x$ such that the  restriction to the preimage of $\widetilde{U}$ can be regarded as the natural projection from the product space of $\widetilde{U}$ and some manifold to $\widetilde{U}$. 
We remark that, if $L_{\widetilde{X}}$ in the above theorem is nef, then $\widetilde{L}$ is strictly nef. 
From this Theorem \ref{countereg} and Corollary \ref{gen_dps}, we can deduce Corollary {\ref{fujino_answer}}. 

We first review on the metric of Mumford's example \ref{fujino_eg} (1). 

\begin{remark}\label{mumford_rmk}
The line bundle $L_Y$ in Example \ref{fujino_eg} (1) admits a smooth Hermitian metric with semi-positive curvature. 
Indeed, we can construct such an metric as follows. 
First we remark that, by the Narasimhan-Seshadri theorem \cite{NS}, there exists a representation $\rho\colon\pi_1(\widetilde{C})\to U(2)$ such that the dual $F^*$ of $F$ is isomorphic to the quotient $(\mathbb{H}\times\mathbb{C}^2)/\pi_1(\widetilde{C})$, where $\pi_1(\widetilde{C})$ acts on $\mathbb{H}\times\mathbb{C}^2$ via $\rho$. Here we denote by $\mathbb{H}$ an upper half-plane $\{x\in\mathbb{C}\mid{\rm Im}\,z>0\}$. 
Thus the fiber-wise Euclidean metric induces a flat metric $h_{F^*}$ of $F^*$ and, on every open set $U\subset \widetilde{C}$, we can choose nice local frame $(s^*_1, s^*_2)$ of $F^*$ such that 
\[|s^*_1|^2_{h_{F^*}}\equiv 1,\ |s^*_2|^2_{h_{F^*}}\equiv 1,\ \langle s^*_1, s^*_2\rangle_{h_{F^*}}\equiv 0\]
holds on $U$. 
We use the function 
\[(w, x)\mapsto [s_1^*(x)+ws_2^*(x)]\in\mathbb{P}(F)\]
as a local coordinates system of $\mathbb{P}(F)$. Then the fiber-wise Fubini-Study metric defines a smooth Hermitian metric $h_{L_Y}$, whose curvature tensor can be computed as
\[\sqrt{-1}\Theta_{h_{L_Y}}=\sqrt{-1}\partial\overline{\partial}\log |s_1^*(x)+ws_2^*(x)|_{F^*}^2=\frac{\sqrt{-1}dw\wedge d\overline{w}}{(1+|w|^2)^2}. \] 
Therefore, this smooth Hermitian metric $h_{L_Y}$ clearly has a semi-positive curvature tensor. 
\end{remark}

\begin{proof}[Proof of Theorem \ref{countereg}]
We fix a singular Hermitian metric $h_{\widetilde{L}}$ of $\widetilde{L}$ with semi-positive curvature and show the existence of a constant $M_{\widetilde{W}}$ for each sufficiently small open set $\widetilde{W}\subset \widetilde{Y}$ such that $(p_1^*h_{L_{\widetilde{X}}})\otimes(p_2^*h_{L_Y})\leq M_{\widetilde{W}}h_{\widetilde{L}}$ holds on $\widetilde{W}$. 

Let us fix a smooth Hermitian metric $h_{\infty}$ of $L_{\widetilde{X}}$. 
We here remark that the tensor product $(p_1^*h_{\infty})\otimes(p_2^*h_Y)$ (or $p_1^*h_{\infty}p_2^*h_Y$, for simplicity) defines a smooth Hermitian metric of $\widetilde{L}$. 
Therefore, there exists a quasi-psh function $\chi\colon \widetilde{Y}\to\mathbb{R}$ such that 
\[h_{\widetilde{L}}=p_1^*h_{\infty}p_2^*h_0e^{-\chi}\]
holds. 
Let us describe our situations in the words of local weight functions. 
Here we use a local coordinates system $(z, w, x)$ of $\widetilde{Y}$, 
where $x$ is a local coordinate of an open set $U\subset \widetilde{C}$, 
$(z, x)$ is a local coordinates system of $\widetilde{X}$ where $z$ is a fiber coordinate, 
and $(w, x)$ is a local coordinates system of $Y$ just as in Remark \ref{mumford_rmk}. 
Let $\varphi_{\infty}$ be the local weight function of $h_{\infty}$ and $\varphi_{\widetilde{L}}$ be the local weight function of $h_{\widetilde{L}}$. Then the equation 
\[\varphi_{\widetilde{L}}(z, w, x)=\varphi_{\infty}(z, 1)+\log (1+|w|^2)+\chi(z, w, x)\]
holds. 
Since $\sqrt{-1}\partial\overline{\partial}\varphi_{\widetilde{L}}\geq 0$ holds from the assumption, 
the inequality 
\[\sqrt{-1}\partial_{z, x}\overline{\partial}_{z, x}\varphi_{\widetilde{L}}(z, w_0, x)\geq 0\]
also holds for all fixed point $w_0\in p^{-1}(U)$, where $p\colon Y\to\widetilde{C}$ is the natural projection and $\partial_{z, x}$ (resp. $\overline{\partial}_{z, x}$) is the operator  $\partial$ (resp. $\overline{\partial}$) with $w_0$ fixed. 
Thus, it can be said that the function $\chi|_{\widetilde{\pi}^{-1}(U)\times\{w_0\}}\colon\widetilde{\pi}^{-1}(U)\to\mathbb{R}; (z, x)\mapsto \chi(z, w_0, x)$ is a $\varphi_{\infty}$-psh function (i.e. $\sqrt{-1}\partial_{z, x}\overline{\partial}_{z, x}(\chi|_{\widetilde{\pi}^{-1}(U)\times\{w_0\}}+\varphi_{\infty})\geq 0$ holds) on $\widetilde{\pi}^{-1}(U)$, where $\widetilde{\pi}\colon\widetilde{X}\to\widetilde{C}$ is the natural surjection. Therefore the function 
\[\widetilde{\chi}(z, x):=\max_{w\in p^{-1}(x)}\chi(z, w, x)\]
is also a $\varphi_{\infty}$-psh function on $\widetilde{\pi}^{-1}(U)$ (Since $\chi$ is quasi-psh and thus it is locally bounded from above, we can use the same argument as the proof of \cite[1.5]{DPS00}). 
As this function $\widetilde{\chi}$ can be regarded as a global function on $\widetilde{X}$, it is clear that the singular Hermitian metric
$h_2:=h_{\infty}e^{-\widetilde{\chi}}$
of $\mathcal{O}_{\widetilde{X}}(\widetilde{D})$ is well-defined singular Hermitian metric with semi-positive curvature. 
Therefore, from the assumption, there exists a positive constant $M_W$ for each sufficiently small open set $W\subset \widetilde{X}$ such that $M_Wh_{\infty}e^{-\widetilde{\chi}}\geq h_{\widetilde{X}}$ holds on $W$. 
Thus we obtain an inequality 
\[
p_1^*h_{\widetilde{X}}p_2^*h_Y\leq M_Wp_1^*(h_{\infty}e^{-\widetilde{\chi}})p_2^*h_Y
\leq M_Wp_1^*h_{\infty}p_2^*h_Ye^{-\chi}=M_Wh_{\widetilde{L}}
\]
holds on $p_1^{-1}(W)$ and therefore we can conclude that the metric $p_1^*h_{\widetilde{X}}p_2^*h_Y$ is a minimal singular metric of $\widetilde{L}$. 
\end{proof}


\end{document}